\documentclass{amsproc}\usepackage{amsmath}

\usepackage{amsfonts}

\theoremstyle{plain}

\newtheorem{definition}{Definition}

\newtheorem{lemma}{Lemma}

\newtheorem{proposition}{Proposition}
\newtheorem{remark}{Remark}

\newtheorem{theorem}{Theorem}
\numberwithin{equation}{section}

\input{tcilatex}

\begin{document}
\title[The second Yamab\'{e} invariant with singularities]{The second Yamab%
\'{e} invariant with singularities}
\author{Mohammed Benalili}
\email{m\_benalili@mail.univ-tlemcen.dz}
\author{Hichem Boughazi}
\curraddr{Universit\'{e} Aboubekr Belka\"{\i}d, Faculty of Sciences, Dept.
of Math. B.P. 119, Tlemcen, Algeria.}
\subjclass{[]}
\keywords{Second Yamabe invariant, singularities, Critical Sobolev growth.}
\dedicatory{Dedicated to the memory of T. Aubin.}

\begin{abstract}
Let $(M,g)$\ be a compact Riemannian manifold of dimension $n\geq 3$.We
suppose that $\ g$\ \ is a metric in the Sobolev space $H_{2}^{p}(M,T^{\ast
}M\otimes T^{\ast }M)$\ with $\ p>\frac{n}{2}$ and there exist a point $\
P\in M$\ \ and $\ \delta >0$\ such that $g$\ \ is smooth in the ball $\
B_{p}(\delta )$. We define the second Yamabe invariant with singularies as
the infimum of the second eigenvalue of \ the singularYamabe operator over a
generalized class of conformal metric to $g$ and of volume $1$. We show that
this operator is attained by a generalized metric, we deduce nodal solutions
to a Yamabe type equation with singularities.
\end{abstract}

\maketitle

\begin{center}
\bigskip 

\bigskip
\end{center}

\section{Introduction}

Let $(M,g)$\ be a compact Riemannian manifold of dimension $n\geq 3$. The
problem of finding a metric conformal to the original one with constant
scalar curvature was first formulated by Yamabe (\cite{9}) and a variational
method was initiated by this latter in an attempt to solve the problem.
Unfortunately or fortunately a serious gap in the Yamabe was pointed out by
Trudinger who addressed the question in the case of non positive scalar
curvature ( \cite{8} ). Aubin (\cite{2}) solved the problem for arbitrary
non locally conformally flat manifolds of dimension $\ n\geq 6$. Finally
Shoen (\cite{7}) solved completely the problem using the positive-mass
theorem founded previously by Shoen himself and Yau . The method to solve
the Yamabe problem could be described as follows: let $u$ be a smooth
positive function and let $\overline{g}=u^{N-2}g$\ \ be a conformal metric
where $\ N=2n/(n-2)$. Up to a multiplying constant, the following equation
is satisfied 
\begin{equation*}
L_{g}(u)=S_{\tilde{g}}|u|^{N-2}u
\end{equation*}%
\ where $\ $%
\begin{equation*}
L_{g}\ =\frac{4(n-1)}{n-2}\Delta +S_{g}\ 
\end{equation*}%
and $S_{g}$ denotes the scalar curvature of $g$. $L_{g}$\ is conformally
invariant called the conformal operator. Consequently, solving the Yamabe
problem is equivalent to find a smooth positive solution to the equation

\begin{equation}
L_{g}(u)=ku^{N-1}  \tag{1}  \label{1}
\end{equation}%
$\ $where $k$\ \ is a constant.

In order to obtain solutions to this equation, Yamabe defined the quantity

\begin{equation*}
%TCIMACRO{\U{3bc} }%
%BeginExpansion
\mu
%EndExpansion
\left( M,g\right) =\underset{u\in C^{\infty }\left( M\right) ,\text{ }u>0}{%
\inf }Y(u)
\end{equation*}

where

\begin{equation*}
Y(u){\LARGE =}\frac{\int_{M}\left( \frac{4(n-1)}{n-2}|\nabla
u|^{2}+S_{g}u^{2}\right) dv_{g}}{{\Large (}\int_{M}|u|^{N}dv_{g}{\Large )}%
^{2/N}}\text{.}
\end{equation*}%
$%
%TCIMACRO{\U{3bc} }%
%BeginExpansion
\mu
%EndExpansion
(M,g)$\ is called the Yamabe invariant, and $Y$ the Yamabe functional. In
the sequel we write $\mu $ instead of $\mu \left( M,g\right) $. Writing the
Euler-Lagrange equation associated to $Y$, we see that there exists a one to
one correspondence between critical points of $Y$\ and solutions of equation
(\ref{1}). In particular, if $u$ is a positive smooth function such that $%
Y(u)=%
%TCIMACRO{\U{3bc} }%
%BeginExpansion
\mu
%EndExpansion
$, then $u$ is a solution of equation (\ref{1})\ and $\ \overline{g}%
=u^{(N-2)}g$\ \ is metric of constant scalar curvature. The key point to
solve the Yamabe problem is the following fundamental results due to Aubin ( 
\cite{2}). Let $S_{n}$\ be the unit euclidean sphere.

\begin{theorem}
Let $(M,g)$ be a compact Riemannian manifold of dimension $n\geq 3$. If $\ 
%TCIMACRO{\U{3bc} }%
%BeginExpansion
\mu
%EndExpansion
\left( M,g\right) <%
%TCIMACRO{\U{3bc} }%
%BeginExpansion
\mu
%EndExpansion
(S_{n})$, then there exists a positive smooth solution $u$ such that $Y(u)=%
%TCIMACRO{\U{3bc} }%
%BeginExpansion
\mu
%EndExpansion
\left( M,g\right) $.
\end{theorem}

This strict inequality $%
%TCIMACRO{\U{3bc} }%
%BeginExpansion
\mu
%EndExpansion
\left( M,g\right) <%
%TCIMACRO{\U{3bc} }%
%BeginExpansion
\mu
%EndExpansion
(S_{n})$ avoids concentration phenomena. Explicitly $%
%TCIMACRO{\U{3bc} }%
%BeginExpansion
\mu
%EndExpansion
(S_{n})=n(n-1)\omega _{n}^{2/n}$\ where $\omega _{n}$\ stands for the volume
of $S_{n}$.

\begin{theorem}
Let $(M,g)$\ be a compact Riemannian manifold of dimension $n\geq 3$. Then 
\begin{equation*}
%TCIMACRO{\U{3bc} }%
%BeginExpansion
\mu
%EndExpansion
\left( M,g\right) \leq 
%TCIMACRO{\U{3bc} }%
%BeginExpansion
\mu
%EndExpansion
(S_{n})\text{.}
\end{equation*}%
\ Moreover, the equality holds if and only if $\ (M,g)$\ is conformally
diffeomorphic to the sphere $\ S_{n}$.
\end{theorem}

Amman and Humbert (\cite{1}) defined the second Yamabe invariant as the
infimum of the second eigenvalue of the Yamabe operator over the conformal
class of the metric $g$\ \ with volume $1$. Their method consists in
considering the spectrum of the operator $L_{g}$ 
\begin{equation*}
spec(L_{g})=\{\lambda _{1,g}\ ,\lambda _{2,g}\ldots \}
\end{equation*}%
\ where the eigenvalues $\lambda _{1,g}\ <\lambda _{2,g}\ldots $\ appear
with their multiplicities. The variational characterization of $\ \lambda
_{1,g}\ $\ is given by

\begin{equation*}
\lambda _{1,g}\ =\underset{u\in C^{\infty }\left( M\right) ,\text{ }u>0}{%
\inf }\frac{\int_{M}\left( \frac{4(n-1)}{n-2}|\nabla
u|^{2}+S_{g}u^{2}\right) dv_{g}}{\int_{M}u^{2}dv_{g}}\text{.}
\end{equation*}%
Let 
\begin{equation*}
\left[ g\right] =\{u^{N-2}g,\text{ }u\in C^{\infty }\left( M\right) ,\text{ }%
u>0\}\ 
\end{equation*}%
Then they defined $\ $the $k^{th}$\ Yamabe invariant with $\ k\in 
%TCIMACRO{\U{2115} }%
%BeginExpansion
\mathbb{N}
%EndExpansion
^{\star }$, \ by

\begin{equation*}
%TCIMACRO{\U{3bc} }%
%BeginExpansion
\mu
%EndExpansion
_{k}=\underset{\overline{g}\in \left[ g\right] }{\inf }\lambda _{k,\overline{%
g}}Vol(M,\tilde{g})^{^{2/n}}\text{.}
\end{equation*}%
With these notations $\ 
%TCIMACRO{\U{3bc} }%
%BeginExpansion
\mu
%EndExpansion
_{1}$ is the Yamabe invariant. They studied the second Yamabe invariant $%
%TCIMACRO{\U{3bc} }%
%BeginExpansion
\mu
%EndExpansion
_{2}$, they found that contrary to the Yamabe invariant, $\ 
%TCIMACRO{\U{3bc} }%
%BeginExpansion
\mu
%EndExpansion
_{2}$\ cannot be attained by a regular metric. In other words, there does
not exist $\overline{g}\in \left[ g\right] $\ , such that%
\begin{equation*}
%TCIMACRO{\U{3bc} }%
%BeginExpansion
\mu
%EndExpansion
_{2}=\lambda _{2,\overline{g}}Vol(M,\tilde{g})^{^{2/n}}\text{.}
\end{equation*}%
\ In order to find minimizers, they enlarged the conformal class to a larger
one. A generalized metric is the one of the form $\ \overline{g}=u^{N-2}g$\
, which is not necessarily positive and smooth, but only $u\in L^{N}(M),$ $%
u\geq 0,u\neq 0$ and where $N=2n/\left( n-2\right) .$\ The definitions of $%
\lambda _{2,\overline{g}}$\ \ and of $\ Vol(M,\overline{g})^{^{2/n}}$can be
extended to generalized metrics. The key points to solve this problem is the
following theorems (\cite{1}).

\begin{theorem}
Let $(M,g)$\ be a compact Riemannian manifold of dimension $\ n\geq 3$,\
then $\ 
%TCIMACRO{\U{3bc} }%
%BeginExpansion
\mu
%EndExpansion
_{2}$\ is attained by a generalized metrics in the following cases.%
\begin{equation*}
%TCIMACRO{\U{3bc} }%
%BeginExpansion
\mu
%EndExpansion
>0\ ,\ 
%TCIMACRO{\U{3bc} }%
%BeginExpansion
\mu
%EndExpansion
_{2}<\left[ (%
%TCIMACRO{\U{3bc} }%
%BeginExpansion
\mu
%EndExpansion
^{n/2}+(%
%TCIMACRO{\U{3bc} }%
%BeginExpansion
\mu
%EndExpansion
(S_{n}))^{n/2}\right] ^{2/n}
\end{equation*}%
and%
\begin{equation*}
%TCIMACRO{\U{3bc} }%
%BeginExpansion
\mu
%EndExpansion
=0,\text{ }\ 
%TCIMACRO{\U{3bc} }%
%BeginExpansion
\mu
%EndExpansion
_{2}<\text{\ }%
%TCIMACRO{\U{3bc} }%
%BeginExpansion
\mu
%EndExpansion
(S_{n})
\end{equation*}

\begin{theorem}
The assumptions of the last theorem are satisfied in the following cases

If $\ (M,g)$\ in not locally conformally flat and, $\ n\geq 11$\ and $\ 
%TCIMACRO{\U{3bc} }%
%BeginExpansion
\mu
%EndExpansion
>0$

If $\ (M,g)$\ in not locally conformally flat and, $%
%TCIMACRO{\U{3bc} }%
%BeginExpansion
\mu
%EndExpansion
=0$\ and $\ n\geq 9$.
\end{theorem}
\end{theorem}

\begin{theorem}
Let $(M,g)$\ be a compact Riemannian manifold of dimension $n\geq 3$, assume
that $\ 
%TCIMACRO{\U{3bc} }%
%BeginExpansion
\mu
%EndExpansion
_{2}$\ is attained by a generalized metric $\ \tilde{g}=u^{N-2}g$\ then
there exist a nodal solution $w\in C^{2,\alpha }(M)$\ of equation 
\begin{equation*}
L_{g}(w)=%
%TCIMACRO{\U{3bc} }%
%BeginExpansion
\mu
%EndExpansion
_{2}|u|^{N-2}w\ 
\end{equation*}%
such that%
\begin{equation*}
\left\vert w\right\vert =u
\end{equation*}%
\ where $\ \alpha \leq N-2$.
\end{theorem}

\bigskip

In (\cite{5}), recently F.Madani studied the Yamabe problem with
singularities when the metric $g$ admits a finite number of points with
singularities and smooth outside these points. Let $(M,g)$\ be a compact
Riemannian manifold of dimension $n\geq 3$, assume that $g$\textit{\ is a
metric in the Sobolev space }$H_{2}^{p}(M,T^{\ast }M\otimes T^{\ast }M)$%
\textit{\ with }$\ p>\frac{n}{2}$\textit{\ and there exist a point }$P\in M$%
\textit{\ and }$\delta >0$\textit{\ such that }$g$\textit{\ \ is smooth in
the ball }$B_{p}(\delta )$\textit{, and let }$(H)$\textit{\ be these
assumptions}. By Sobolev's embedding, we have for $p>\frac{n}{2}$, $%
H_{2}^{p}(M,T^{\ast }M\otimes T^{\ast }M)\subset C^{1-\left[ n/p\right]
,\beta }\left( M,T^{\ast }M\otimes T^{\ast }M\right) $, where $\left[ n/p%
\right] $ denotes the entire part of $n/p$. Hence the metric satisfying
assumption $(H)$ is of class $C^{1-\left[ \frac{n}{p}\right] ,\beta }$ with $%
\beta \in \left( 0,1\right) $ provided that $p>n$. The Christoffels symbols
belong to $\ H_{1}^{p}\left( M\right) $ ( to $C^{o}\left( M\right) $ in case 
$p>n$), the Riemannian curvature tensor, the Ricci tensor and scalar
curvature are in $L^{p}(M)$. F. Madani proved under the assumption ($H$) the
existence of a metric $\overline{g}=u^{N-2}g$\ \ conformal to $\ g$\ \ such
that $\ u\in H_{2}^{p}(M)$, $u>0$\ and the scalar curvature $\ S_{\overline{g%
}}$\ \ of $\ \overline{g}$\ \ is constant and $(M,g)$\ \ is not conformal to
the round sphere. Madani proceeded as follows: let $\ $\ $u\in H_{2}^{p}(M)$%
, $u>0$\ be a function and $\overline{g}=u^{N-2}g$\ a particular conformal
metric where $N=2n/(n-2)$.Then, multiplying$\ u$\ by a constant, the
following equation is satisfied

\begin{equation*}
L_{g}u=\frac{n-2}{4(n-1)}S_{\tilde{g}}|u|^{N-2}u
\end{equation*}%
where 
\begin{equation*}
L_{g}\ =\Delta _{g}+\frac{n-2}{4(n-1)}S_{g}
\end{equation*}%
and the scalar curvature $S_{g}$ is in $\ L^{p}(M)$. Moreover $L_{g}$\ is
weakly conformally invariant hence solving the singular Yamabe problem is
equivalent to find a positive solution $u\in H_{2}^{p}(M)$\ \ of \ 
\begin{equation}
L_{g}u=k|u|^{N-2}u\text{ \ \ \ \ \ \ \ }  \tag{2}  \label{2}
\end{equation}%
where $k$ \ is a constant. In order to obtain solutions of equation (\ref{2}%
) we define the quantity

\begin{equation*}
%TCIMACRO{\U{3bc} }%
%BeginExpansion
\mu
%EndExpansion
=\underset{u\in H_{2}^{p}\left( M\right) ,\ \ u>0}{\inf }Y(u)
\end{equation*}%
where

\begin{equation*}
Y(u)=\frac{\int_{M}\left( |\nabla u|^{2}+\frac{(n-2)}{4(n-1)}%
S_{g}u^{2}\right) dv_{g}}{(\int_{M}|u|^{N}dv_{g})^{2/N}}\text{.}
\end{equation*}%
$\ 
%TCIMACRO{\U{3bc} }%
%BeginExpansion
\mu
%EndExpansion
$\ is called Yamabe invariant with singularities. Writing the Euler-Lagrange
equation associated to $Y$ , we see that there exists a one to one
correspondence between critical points of $Y$\ \ and solutions of equation (%
\ref{2}). In particular, if $\ u\in H_{2}^{p}$\ \ is a positive function
which minimizes $Y$, then $\ u$\ \ is a solution of equation(\ref{2})\ and $%
\ \overline{g}=u^{N-2}g$\ is a metric of constant scalar curvature and $%
%TCIMACRO{\U{3bc} }%
%BeginExpansion
\mu
%EndExpansion
$ is attained by a particular conformal metric. The key points to solve the
above problem are the following theorems (\cite{5}).

\begin{theorem}
If $p>n/2$\ and $\ 
%TCIMACRO{\U{3bc} }%
%BeginExpansion
\mu
%EndExpansion
<K^{-2}$then equation\ref{2} admits a positive solution $u\in
H_{2}^{p}\left( M\right) \subset C^{1-[n/p],\beta }(M)$\ ; $[n/p]$\ is the
integer part of $\ n/p$, $\beta \in (0,1)$\ which minimizes $Y$, where $%
K^{2}=\frac{4}{n(n-1)}\omega _{n}^{-2/n}$ with $\omega _{n}$\ denotes the
volume of $S_{n}$. \ If $p>n$ , then $u\in H_{2}^{p}\left( M\right) \subset
C^{1}\left( M\right) $.
\end{theorem}

\begin{theorem}
Let $(M,g)$ be a compact Riemannian manifold of dimension $n\geq 3$. $\ g$\
is a metric which satisfies the assumption ($H$). If $(M,g)$\ is not
conformal to the sphere $S_{n}$ with the standard Riemannian structure then%
\begin{equation*}
%TCIMACRO{\U{3bc} }%
%BeginExpansion
\mu
%EndExpansion
<K^{-2}
\end{equation*}
\end{theorem}

\begin{theorem}
\label{6} (\cite{5}) On an n -dimensional compact Riemannian manifold $(M,g)$%
, if $u\geq 0$\ is a non trivial weak solution in $\ H_{1}^{2}\left(
M\right) $ of equation $\ \Delta u+hu=0$, with $\ h\in L^{p}(M$) \ and $\
p>n/2$, then $u\in C^{1-[n/p],\beta }$\ and $u>0$; $[n/p]$\ is the integer
part of $\ n/p$\ and $\beta \in (0,1)$.
\end{theorem}

For\ regularity\ argument we\ need\ the\ following results

\begin{lemma}
\label{lem7} Let $u\in L_{+}^{N}(M)$\ and $v$\ $\in $ $H_{1}^{2}(M)$ a weak
solution to $\ L_{g}(v)=u^{N-2}v$, then 
\begin{equation*}
v\in L^{N+\epsilon }(M)
\end{equation*}%
\ for some $\varepsilon >0$.
\end{lemma}

The proof is the same as in (\cite{5}) with some modifications. As a
consequence of Lemma \ref{7}, $v\in L^{s}(M)$, $\forall s\geq 1$.

\begin{proposition}
\label{pro1} If $g$\ $\in H_{2}^{p}(M,T^{\ast }M\otimes T^{\ast }M)$ is a
Riemannian metric on $M$\ with $p>n/2$. If $\ \overline{g}=u^{N-2}g$\ is a
conformal metric to $\ g$\ \ such that $\ u\in H_{2}^{p}\left( M\right) $, $%
u>0$\ then $\ L_{g}$\ is weakly conformally invariant, which means that $%
\forall v\in H_{1}^{2}\left( M\right) $, $|u|^{N-1}L_{\overline{g}%
}(v)=L_{g}(uv)\ $weakly. Moreover if $%
%TCIMACRO{\U{3bc} }%
%BeginExpansion
\mu
%EndExpansion
>0$,\ then $\ L_{g}$ \ is coercive and invertible.
\end{proposition}

\bigskip

In this paper, let $(M,g)$\ be a compact Riemannian manifold of dimension $%
n\geq 3$. We suppose that $g$\ is a metric in the Sobolev space $\
H_{2}^{p}(M,T^{\ast }M\otimes T^{\ast }M)$\ with $\ p>n/2$\ and there exist
a point $\ P\in M$\ \ and $\ \delta >0$\ such that $g$\ \ is smooth in the
ball $B_{P}(\delta )$ and we call these assumptions the condition ($H$).

In the smooth case the operator $L_{g}$\ is an elliptic operator on $M$
self- adjoint, and has a discrete spectrum $\ Spec(L_{g})=\{\lambda _{1,g}\
,\lambda _{2,g},\ldots \}$, where the eigenvalues $\lambda _{1,g}\ <\lambda
_{2,g}\ldots $\ appear with their multiplicities. \ These properties remain
valid also in the case where $S_{g}\in L^{p}\left( M\right) $. The
variational characterization of $\ \lambda _{1,g}$ is given by

\begin{equation*}
\lambda _{1,g}\ =\underset{u\in H_{1}^{2},u>0}{\inf }\frac{\int_{M}\left(
|\nabla u|^{2}+\frac{(n-2)}{4(n-1)}S_{g}u^{2}\right) dv_{g}}{%
\int_{M}u^{2}dv_{g}}
\end{equation*}

Let $\ \left[ g\right] =\{u^{N-2}g$\ : $u\in H_{2}^{p}$\ and $u>0\}$,\ Let $%
k\in 
%TCIMACRO{\U{2115} }%
%BeginExpansion
\mathbb{N}
%EndExpansion
^{\ast }$, we define the $k^{th}$Yamabe invariant with singularities$\ 
%TCIMACRO{\U{3bc} }%
%BeginExpansion
\mu
%EndExpansion
_{k}$\ as

\begin{equation*}
%TCIMACRO{\U{3bc} }%
%BeginExpansion
\mu
%EndExpansion
_{k}=\underset{\overline{g}\in \left[ g\right] }{\inf }\lambda _{k,\overline{%
g}}Vol(M,\tilde{g})^{^{2/n}}
\end{equation*}%
with these notations, $%
%TCIMACRO{\U{3bc} }%
%BeginExpansion
\mu
%EndExpansion
_{1}$ is the first Yamabe invariant with singularities.

In this work we are concerned with $\mu _{2}$. In order to find minimizers
to $\mu _{2}$ we extend the conformal class to a larger one consisting of
metrics of the form $\overline{g}=u^{N-2}g$ where $u$ is no longer
necessarily in $H_{2}^{p}\left( M\right) $ and positive but$\ u\in
L_{+}^{N}\left( M\right) =\left\{ L^{N}(M),u\geq 0,u\neq 0\right\} $ such
metrics will be called for brevity generalized metrics. First we are going
to show that if the singular Yamabe invariant $%
%TCIMACRO{\U{3bc} }%
%BeginExpansion
\mu
%EndExpansion
\geq 0$ then $%
%TCIMACRO{\U{3bc} }%
%BeginExpansion
\mu
%EndExpansion
_{1}$ it is exactly $\mu $ next we consider $\mu _{2}$, \ $%
%TCIMACRO{\U{3bc} }%
%BeginExpansion
\mu
%EndExpansion
_{2}$ is attained by a conformal generalized metric.

Our main results state as follows:

\begin{theorem}
Let $\ (M,g)$\ be a compact Riemannian manifold of dimension $n\geq 3$ .We
suppose that $g$\ is a metric in the Sobolev space $\ H_{2}^{p}(M,T^{\ast
}M\otimes T^{\ast }M)$\ with $\ p>n/2$. There exist a point$\ P\in M$\ \ and 
$\ \delta >0$\ such that $g$\ is smooth in the ball $B_{P}(\delta )$, then
\end{theorem}

\begin{equation*}
%TCIMACRO{\U{3bc} }%
%BeginExpansion
\mu
%EndExpansion
_{1}=%
%TCIMACRO{\U{3bc} }%
%BeginExpansion
\mu
%EndExpansion
\text{.}
\end{equation*}

\begin{theorem}
Let $\ (M,g)$\ be a compact Riemannian manifold of dimension $n\geq 3$, we
suppose that $\ g$\ \ is a metric in the Sobolev space $\
H_{2}^{p}(M,T^{\ast }M\otimes T^{\ast }M)$\ with $\ p>n/2$\ .There exist a
point $\ P\in M$\ \ and $\ \delta >0$\ such that $g$\ \ is smooth in the
ball $B_{P}(\delta )$ .Assume that $\ 
%TCIMACRO{\U{3bc} }%
%BeginExpansion
\mu
%EndExpansion
_{2}$\ is attained by a metric $\ \overline{g}=u^{N-2}g$\ where $u\in
L_{+}^{N}\left( M\right) $, then there exist a nodal solution $w\in C^{1-%
\left[ n/p\right] ,\beta }$, $\beta \in \left( 0,1\right) $,\ of equation 
\begin{equation*}
L_{g}w=%
%TCIMACRO{\U{3bc} }%
%BeginExpansion
\mu
%EndExpansion
_{2}u^{N-2}w\ \text{.}
\end{equation*}%
Moreover there exist$\ $real numbers $a,b>0$\ such that 
\begin{equation*}
u=aw_{+}+bw_{-}
\end{equation*}%
with $w_{+}=\sup (w,0)$\ \ and $w_{-}=\sup $($-w,0$) .
\end{theorem}

\begin{theorem}
Let $\ (M,g)$\ be a compact Riemannian manifold of dimension $n\geq 3$,
suppose that $\ g$\ \ is a metric in the Sobolev space $\
H_{2}^{p}(M,T^{\ast }M\otimes T^{\ast }M)$\ with $\ p>n/2$. There exist a
point $\ P\in M$\ \ and $\ \delta >0$\ such that $g$\ is smooth in the ball $%
B_{P}(\delta )$ then $%
%TCIMACRO{\U{3bc} }%
%BeginExpansion
\mu
%EndExpansion
_{2}$\ is attained by a generalized metric\ in the following cases:

If $\ (M,g)$\ in not locally conformally flat and, $\ n\geq 11$\ and $\ 
%TCIMACRO{\U{3bc} }%
%BeginExpansion
\mu
%EndExpansion
>0$

If $\ (M,g)$\ in not locally conformally flat and, $%
%TCIMACRO{\U{3bc} }%
%BeginExpansion
\mu
%EndExpansion
=0$\ and $\ n\geq 9$.
\end{theorem}

\section{\protect\bigskip Generalized\ metrics\ and\ the\ Euler-Lagrange\
equation}

Let 
\begin{equation*}
L_{+}^{N}(M)=\left\{ u\in L_{+}^{N}(M)\text{: }u\geq 0,u\neq 0\right\}
\end{equation*}%
where $N=\frac{2n}{n-2}$.

As in (\cite{1})

\begin{definition}
For all $u\in L_{+}^{N}(M)$, we define $Gr_{k}^{u}(H_{1}^{2}\left( M\right)
) $ to be the set of all $k-$dimensional subspaces of $H_{1}^{2}\left(
M\right) $ with span($v_{1},v_{2},...,v_{k})\in Gr_{k}^{u}(H_{1}^{2}\left(
M\right) )$ if and only if $v_{1},v_{2},...,v_{k}$ are linearly independent
on $M-u^{-1}(0).$
\end{definition}

Let $\ (M,g)$\ be a compact Riemannian manifold of dimension $n\geq 3$. For
a generalized metric$\ \overline{g}$\ conformal to $g$, we define%
\begin{equation*}
\lambda _{k,\overline{g}}=\underset{V\in Gr_{k}^{u}(H_{1}^{2}\left( M\right)
)}{\inf }\underset{v\in V}{\text{ }\sup }\frac{\int_{M}vL_{g}(v)dv_{g}}{%
\int_{M}|u|^{N-2}v^{2}dv_{g}}\text{.}
\end{equation*}%
We quote the following regularity theorem

\begin{theorem}
\label{th6} \cite{7} On a n -dimensional compact Riemannian manifold $(M,g)$%
, if $u\geq 0$\ is a non trivial weak solution in $\ H_{1}^{2}\left(
M\right) $ of the equation $\ $%
\begin{equation*}
\Delta u+hu=cu^{N-1}
\end{equation*}%
with $\ h\in L^{p}(M$) \ and $\ p>n/2$, then 
\begin{equation*}
u\in H_{2}^{p}(M)\subset C^{1-[n/p],\beta }(M)
\end{equation*}%
\ and $u>0$, where $[n/p]$ denotes the integer part of $\ n/p$\ and $\beta
\in (0,1)$.
\end{theorem}

\begin{proposition}
\label{pro2} Let $(v_{m})$ be a sequence in $H_{1}^{2}\left( M\right) $ such
that $v_{m}\rightarrow v$\ strongly in $L^{2}\left( M\right) $, then for all
any $u\in L_{+}^{N}(M)$%
\begin{equation*}
\dint\limits_{M}u^{N-2}(v^{2}-v_{m}^{2})\ dv_{g}\rightarrow 0\text{.}
\end{equation*}
\end{proposition}

\begin{proof}
The proof is the same as in (\cite{3}).
\end{proof}

\begin{proposition}
\label{prop2} If $\ 
%TCIMACRO{\U{3bc} }%
%BeginExpansion
\mu
%EndExpansion
>0$, then for all $u\in L_{+}^{N}(M)$, there exist two functions $v,w$\ in $%
H_{1}^{2}\left( M\right) $ with $\ v\geq 0$\ satisfying in the weak sense
the equations%
\begin{equation}
L_{g}v=\lambda _{1,\overline{g}}u^{N-2}v\text{ }  \tag{7}  \label{7}
\end{equation}%
and 
\begin{equation}
L_{g}w=\lambda _{2,\overline{g}}u^{N-2}w  \tag{8}  \label{8}
\end{equation}%
Moreover we can choose $v$ and $w$\ such that%
\begin{equation}
\dint\limits_{M}u^{N-2}w^{2}dv_{g}=\dint\limits_{M}u^{N-2}v^{2}dv_{g}=1\ \ \
\ \text{and\ \ \ \ \ \ \ }\dint\limits_{M}u^{N-2}wvdv_{g}=0\text{. }  \tag{9}
\label{9}
\end{equation}
\end{proposition}

\bigskip

\begin{proof}
Let $\left( v_{m}\right) _{m}$\ be a minimizing sequence for $\lambda _{1,%
\tilde{g}}$\ i.e. a sequence $v_{m}\in H_{1}^{2}$\ such that 
\begin{equation*}
\lim_{m}\frac{\int_{M}v_{m}L_{g}(v_{m})dv_{g}}{%
\int_{M}|u|^{N-2}v_{m}^{2}dv_{g}}=\lambda _{1,\tilde{g}}
\end{equation*}

It is well know that $(|v_{m}|)_{m}$\ is also minimizing sequence. Hence we
can assume that $v_{m}\geq 0$. If we normalize $\left( v_{m}\right) _{m}$\ by%
\begin{equation*}
\int_{M}|u|^{N-2}v_{m}^{2}dv_{g}=1\text{.}
\end{equation*}%
Now by the fact that $L_{g}$ is coercive%
\begin{equation*}
c\Vert v_{m}\Vert _{H_{1}^{2}}\leq \int_{M}v_{m}L_{g}(v_{m})dv_{g}\leq
\lambda _{1,\tilde{g}}+1\text{.}
\end{equation*}%
$\left( v_{m}\right) _{m}$ is bounded in $H_{1}^{2}\left( M\right) $\ and
after restriction to a subsequence we may assume that there exist $v\in
H_{1}^{2}\left( M\right) $, $v\geq 0$\ such that $v_{m}\rightarrow v$\
weakly in $H_{1}^{2}\left( M\right) $, strongly in $L^{2}\left( M\right) $\
and almost everywhere in $M$, then $v$ satisfies in the sense of
distributions%
\begin{equation*}
L_{g}v=\lambda _{1,\overline{g}}u^{N-2}v\text{.}
\end{equation*}

If $u\in H_{2}^{p}(M)$ $\subset C^{1-\left[ \frac{n}{p}\right] ,\beta }$ then

\begin{equation*}
\dint\limits_{M}u^{N-2}(v^{2}-v_{m}^{2})\ dv_{g}\rightarrow 0\text{ }
\end{equation*}%
and%
\begin{equation*}
\dint\limits_{M}u^{N-2}v^{2}dv_{g}=1\text{.}
\end{equation*}

Then $v$\ \ is not trivial and nonnegative minimizer of $\lambda _{1,%
\overline{g}}$, by Lemma \ref{7} 
\begin{equation*}
h=S_{g}-\lambda _{1,\overline{g}}u^{N-2}\in L^{p}(M)\ 
\end{equation*}%
and by Theorem \ref{6} 
\begin{equation*}
v\in C^{1-\left[ \frac{n}{p}\right] ,\beta }\left( M\right) \text{ }
\end{equation*}%
and%
\begin{equation*}
\text{ }v>0\text{.}
\end{equation*}

If $\ u\in L_{+}^{N}(M)$, by Proposition \ref{pro2} , we get

\begin{equation*}
\dint\limits_{M}u^{N-2}(v^{2}-v_{m}^{2})\ dv_{g}\rightarrow 0\ \text{\ }
\end{equation*}%
so$\ \ $%
\begin{equation*}
\dint\limits_{M}u^{N-2}v^{2}dv_{g}=1\text{.}
\end{equation*}

$v$\ is a non negative minimizer in $H_{1}^{2}$ of $\lambda _{1,\overline{g}%
} $\ such that $\dint\limits_{M}u^{N-2}v^{2}dv_{g}=1$.

Now consider the set%
\begin{equation*}
E=\{w\in H_{1}^{2}\text{: such that }u^{\frac{N-2}{2}}w\neq 0\ \text{and}%
\dint\limits_{M}u^{N-2}wvdv_{g}=0\}
\end{equation*}%
and define%
\begin{equation*}
\lambda _{2,g}^{\prime }=\inf_{w\in E}\frac{\int_{M}wL_{g}(w)dv_{g}}{%
\int_{M}|u|^{N-2}w^{2}dv_{g}}\text{.}
\end{equation*}%
Let ($w_{m}$)\ be a minimizing sequence for $\lambda _{2,g}^{\prime }$\ i.e.
a sequence $w_{m}\in E$ such that

\begin{equation*}
\lim_{m}\frac{\int_{M}w_{m}L_{g}(w_{m})dv_{g}}{%
\int_{M}|u|^{N-2}w_{m}^{2}dv_{g}}=\lambda _{2,g}^{\prime }\text{.}
\end{equation*}

The same arguments lead to a minimizer $w$ \ to $\lambda _{2,g}^{\prime }$%
with $\dint\limits_{M}u^{N-2}w^{2}=1$.

Now writing 
\begin{equation*}
\dint\limits_{M}u^{N-2}wvdv_{g}=\dint\limits_{M}u^{N-2}v(w-w_{m})dv_{g}+%
\dint\limits_{M}u^{N-2}w_{m}vdv_{g}
\end{equation*}

and taking account of $\dint\limits_{M}u^{N-2}w_{m}vdv_{g}=0$ and the fact
that $w_{m}\rightarrow w$\ weakly in $L^{N}\left( M\right) $ and since $%
u^{N-2}v\in L^{\frac{N}{N-1}}\left( M\right) $, we infer that%
\begin{equation*}
\dint\limits_{M}u^{N-2}wvdv_{g}=0\text{.}
\end{equation*}

Hence (\ref{8}) and (\ref{9}) are that satisfied with $\lambda
_{2,g}^{\prime }$\ instead of $\lambda _{2,\overline{g}}$.
\end{proof}

\begin{proposition}
We have 
\begin{equation*}
\lambda _{2,g}^{\prime }=\lambda _{2,\overline{g}}\text{.}
\end{equation*}%
.
\end{proposition}

\begin{proof}
The Proof is the same as in (\cite{3}) so we omit it.
\end{proof}

\begin{remark}
\label{rem1} If $p>n$ then$\ \ u\in H_{2}^{p}\left( M\right) \subset
C^{1}\left( M\right) $, by Theorem \ref{9}, $v$ and $w\in C^{1}\left(
M\right) $ with $\ v>0$.
\end{remark}

\begin{remark}
\label{rem2} If $p>n$ then$\ \ u\in H_{2}^{p}\left( M\right) \subset
C^{1}\left( M\right) $\ and $\ \lambda _{2,\overline{g}}=\lambda _{1,%
\overline{g}}$, we see that $\ |w|$\ is a minimizer for the functional
associated to $\lambda _{1,\overline{g}}$, then $|w|$\ satisfies the same
equation as$\ v$\ \ and by Theorem \ref{9} $\ $we get $\left\vert
w\right\vert >0$, this contradicts relation (\ref{9})\ necessarily 
\begin{equation*}
\ \lambda _{2,\overline{g}}>\lambda _{1,\overline{g}}\text{.}
\end{equation*}
\end{remark}

\section{Variational\ characterization\ and\ existence\ of\ \ $\protect\mu %
_{1}$}

In this section we need the following results

\begin{theorem}
\label{th5} Let $\left( M,g\right) $ be a compact n -dimensional Riemannian
manifold. For any $\varepsilon >0,$\ there exists $A(\varepsilon )>0$\ such
that $\forall u\in H_{1}^{2}\left( M\right) $,%
\begin{equation*}
\Vert u\Vert _{N}^{2}\leq (K^{2}+\varepsilon )\Vert \nabla u\Vert
_{2}^{2}+A(\varepsilon )\Vert u\Vert _{2}^{2}
\end{equation*}%
where $N=2n/(n-4)$\ and $K^{2}=4/(n(n-2))$\ $\omega _{n}^{\frac{-2}{n}}$. $%
\omega _{n}$ is the volume of the round sphere\ $S_{n}$.
\end{theorem}

Let $\ [g]=\{u^{N-2}g$\ : $u\in H_{2}^{p}\left( M\right) $\ and $u>0\},$\ we
define the first singular Yamabe invariant $\mu _{1}$\ as

\begin{equation*}
%TCIMACRO{\U{3bc} }%
%BeginExpansion
\mu
%EndExpansion
_{1}=\underset{\overline{g}\in \lbrack g]}{\inf }\lambda _{1,\overline{g}%
}Vol(M,\tilde{g})^{^{2/n}}
\end{equation*}

then we get

\begin{equation*}
%TCIMACRO{\U{3bc} }%
%BeginExpansion
\mu
%EndExpansion
_{1}=\underset{u\in H_{2}^{p},V\in Gr_{1}^{u}(H_{1}^{2})}{\inf }\underset{%
v\in V}{\text{ }\sup }\frac{\int_{M}vL_{g}(v)dv_{g}}{%
\int_{M}|u|^{N-2}v^{2}dv_{g}}(\dint\limits_{M}u^{N}dv_{g})^{\frac{2}{n}}%
\text{.}
\end{equation*}

\begin{lemma}
\label{lem6} We have 
\begin{equation*}
%TCIMACRO{\U{3bc} }%
%BeginExpansion
\mu
%EndExpansion
_{1}\leq 
%TCIMACRO{\U{3bc} }%
%BeginExpansion
\mu
%EndExpansion
<K^{-2}\text{.}
\end{equation*}%
.
\end{lemma}

\begin{proof}
If $\ p\geq 2n/(n+2)$, the embedding $H_{2}^{p}\left( M\right) \subset
H_{1}^{2}$ $\left( M\right) $ is true, so%
\begin{equation*}
%TCIMACRO{\U{3bc} }%
%BeginExpansion
\mu
%EndExpansion
_{1}=\underset{u\in H_{2}^{p},V\in Gr_{1}^{u}(H_{1}^{2}\left( M\right) )}{%
\inf }\underset{v\in V}{\text{ }\sup }\frac{\int_{M}vL_{g}(v)dv_{g}}{%
\int_{M}|u|^{N-2}v^{2}dv_{g}}(\dint\limits_{M}u^{N}dv_{g})^{\frac{2}{n}}
\end{equation*}%
\begin{equation*}
\leq \underset{u\in H_{2}^{p},V\in Gr_{1}^{u}(H_{2}^{p}\left( M\right) )}{%
\inf }\underset{v\in V}{\text{ }\sup }\frac{\int_{M}vL_{g}(v)dv_{g}}{%
\int_{M}|u|^{N-2}v^{2}dv_{g}}(\dint\limits_{M}u^{N}dv_{g})^{\frac{2}{n}}%
\text{.}
\end{equation*}%
in particular for $p>\frac{n}{2}$ and $u=v$ \ we get \ 
\begin{equation*}
%TCIMACRO{\U{3bc} }%
%BeginExpansion
\mu
%EndExpansion
_{1}\leq \underset{v\in H_{2}^{P},V\in Gr_{1}^{u}(H_{2}^{P}\left( M\right) )}%
{\inf }\underset{v\in V}{\text{ }\sup }\frac{\int_{M}vL_{g}(v)dv_{g}}{%
\int_{M}|v|^{N-2}v^{2}dv_{g}}(\dint\limits_{M}v^{N}dv_{g})^{\frac{2}{n}}=%
%TCIMACRO{\U{3bc}}%
%BeginExpansion
\mu%
%EndExpansion
\end{equation*}%
i.e%
\begin{equation*}
%TCIMACRO{\U{3bc} }%
%BeginExpansion
\mu
%EndExpansion
_{1}\leq 
%TCIMACRO{\U{3bc} }%
%BeginExpansion
\mu
%EndExpansion
<K^{-2}\text{.}
\end{equation*}
\end{proof}

\begin{theorem}
If $\ \ 
%TCIMACRO{\U{3bc} }%
%BeginExpansion
\mu
%EndExpansion
>0$, there exits conform metric $\overline{g}=u^{N-2}g$\ which minimizes $\ 
%TCIMACRO{\U{3bc} }%
%BeginExpansion
\mu
%EndExpansion
_{1}$.
\end{theorem}

\begin{proof}
The proof will take several steps.

\textit{Step1.}

We study a sequence of metrics $g_{m}=u_{m}^{N-2}g$\ \ with $\ u_{m}\in
H_{2}^{p}\left( M\right) $, $u_{m}>0$ which minimize $\ 
%TCIMACRO{\U{3bc} }%
%BeginExpansion
\mu
%EndExpansion
_{1}$\ i.e. a sequence of metrics such that

\begin{equation*}
%TCIMACRO{\U{3bc} }%
%BeginExpansion
\mu
%EndExpansion
_{1}=\lim_{m}\lambda _{1,m}(Vol(M,g_{m})^{2/n}\text{.}
\end{equation*}

Without loss of generality, we may assume that $Vol(M,g_{m})=1$\ i.e. 
\begin{equation*}
\dint\limits_{M}u_{m}^{N}dv_{g}=1\text{.}
\end{equation*}

In particular, the sequence of functions $u_{m}$\ is bounded in $L^{N}\left(
M\right) $\ and there exists $u\in L^{N}\left( M\right) $, $u\geq 0$ \ such
that $u_{m}\rightarrow u$\ \ weakly in $\ L^{N}\left( M\right) $. We are
going to prove that the generalized metric $\ u^{N-2}g$\ \ minimizes $%
%TCIMACRO{\U{3bc} }%
%BeginExpansion
\mu
%EndExpansion
_{1}$. Proposition \ref{prop2} implies the existence of $\ $a sequence $%
\left( v_{m}\right) $ of class $H_{1}^{2}\left( M\right) $, $v_{m}>0$\ \
such that

\begin{equation*}
L_{g}(v_{m})=\lambda _{1,m}u_{m}^{N-2}v_{m}
\end{equation*}%
\ and $\ $%
\begin{equation*}
\int_{M}u_{m}^{N-2}v_{m}^{2}dv_{g}=1\text{.}
\end{equation*}

now since $\mu >0$, by Proposition \ref{pro1}, $L_{g}$ is coercive and we
infer that%
\begin{equation*}
c\Vert v_{m}\Vert _{H_{1}^{2}}\leq \int_{M}v_{m}L_{g}(v_{m})dv_{g}=\lambda
_{1,m}\leq 
%TCIMACRO{\U{3bc} }%
%BeginExpansion
\mu
%EndExpansion
_{1}+1\text{.}
\end{equation*}

The sequence $\left( v_{m}\right) _{m}$ is bounded in $H_{1}^{2}\left(
M\right) $, we can find $v\in H_{1}^{2}\left( M\right) $, $v\geq 0$ such
that $v_{m}\rightarrow v$\ weakly in $H_{1}^{2}\left( M\right) $. Together
with the weak convergence of $\left( u_{m}\right) _{m}$, we obtain in the
sense of distributions 
\begin{equation*}
L_{g}(v)=%
%TCIMACRO{\U{3bc} }%
%BeginExpansion
\mu
%EndExpansion
_{1}u^{N-2}v\text{.}
\end{equation*}

\textit{Step2}.

Now we are going to show that $v_{m}\rightarrow v$\ strongly in $%
H_{1}^{2}\left( M\right) $.

We put 
\begin{equation*}
z_{m}=v_{m}-v\ 
\end{equation*}%
then $z_{m}\rightarrow 0$ weakly in $H_{1}^{2}\left( M\right) $ and strongly
in $L^{q}\left( M\right) $ \ with $q<N$, and writing

\begin{equation*}
\int_{M}\left\vert \nabla v_{m}\right\vert ^{2}dv_{g}=\int_{M}\left\vert
\nabla z_{m}\right\vert ^{2}dv_{g}+\int_{M}\left\vert \nabla v\right\vert
^{2}dv_{g}+2\int_{M}\nabla z_{m}\nabla vdv_{g}
\end{equation*}

we see that%
\begin{equation*}
\int_{M}\left\vert \nabla v_{m}\right\vert ^{2}dv_{g}=\int_{M}\left\vert
\nabla z_{m}\right\vert ^{2}dv_{g}+\int_{M}\left\vert \nabla v\right\vert
^{2}dv_{g}+o(1)\text{.}
\end{equation*}

Now because of $2p/(p-1)<N$ , we have 
\begin{equation*}
\int_{M}\frac{n-2}{4(n-1)}S_{g}(v_{m}-v)^{2}dv_{g}\leq \frac{n-2}{4(n-1)}%
\Vert S_{g}\Vert _{p}\Vert v_{m}-v\Vert _{\frac{2p}{p-1}}^{2}\rightarrow 0
\end{equation*}

so%
\begin{equation*}
\int_{M}\frac{n-2}{4(n-1)}S_{g}v_{m}{}^{2}dv_{g}=\int_{M}\frac{n-2}{4(n-1)}%
S_{g}v^{2}dv_{g}+0(1)
\end{equation*}%
and

\begin{equation*}
\int_{M}\left\vert \nabla v_{m}\right\vert ^{2}dv_{g}+\int_{M}\frac{n-2}{%
4(n-1)}S_{g}(v_{m})^{2}dv_{g}
\end{equation*}

\begin{equation*}
=\int_{M}\left\vert \nabla z_{m}\right\vert ^{2}dv_{g}+\int_{M}\left\vert
\nabla v\right\vert ^{2}dv_{g}+\int_{M}\frac{n-2}{4(n-1)}%
S_{g}(v)^{2}dv_{g}+0(1)\text{.}
\end{equation*}

Then%
\begin{equation*}
\int_{M}v_{m}L_{g}v_{m}dv_{g}=\int_{M}\left\vert \nabla z_{m}\right\vert
^{2}dv_{g}+\int_{M}\left\vert \nabla v\right\vert ^{2}dv_{g}+\int_{M}\frac{%
n-2}{4(n-1)}S_{g}v^{2}dv_{g}+0(1)
\end{equation*}

And by the definition of $%
%TCIMACRO{\U{3bc} }%
%BeginExpansion
\mu
%EndExpansion
$ and Lemma \ref{lem6} we get 
\begin{equation*}
\int_{M}\left\vert \nabla v\right\vert ^{2}dv_{g}+\int_{M}\frac{n-2}{4(n-1)}%
S_{g}(v)^{2}dv_{g}\geq 
%TCIMACRO{\U{3bc} }%
%BeginExpansion
\mu
%EndExpansion
(\dint\limits_{M}v^{N}dv_{g})^{\frac{2}{N}}\geq 
%TCIMACRO{\U{3bc} }%
%BeginExpansion
\mu
%EndExpansion
_{1}(\dint\limits_{M}v^{N}dv_{g})^{\frac{2}{N}}
\end{equation*}

then

\begin{equation*}
\int_{M}v_{m}L_{g}(v_{m})dv_{g}\geq \int_{M}\left\vert \nabla
z_{m}\right\vert ^{2}dv_{g}+%
%TCIMACRO{\U{3bc} }%
%BeginExpansion
\mu
%EndExpansion
_{1}(\dint\limits_{M}v^{N}dv_{g})^{\frac{2}{N}}+0(1)\text{.}
\end{equation*}

And since 
\begin{equation*}
\int_{M}v_{m}L_{g}(v_{m})dv_{g}=\lambda _{1,m}\leq 
%TCIMACRO{\U{3bc} }%
%BeginExpansion
\mu
%EndExpansion
_{1}+0(1)
\end{equation*}

and%
\begin{equation*}
\int_{M}\left\vert \nabla z_{m}\right\vert ^{2}dv_{g}+%
%TCIMACRO{\U{3bc} }%
%BeginExpansion
\mu
%EndExpansion
_{1}(\dint\limits_{M}v^{N}dv_{g})^{\frac{2}{N}}\leq 
%TCIMACRO{\U{3bc} }%
%BeginExpansion
\mu
%EndExpansion
_{1}+0(1)
\end{equation*}%
i.e%
\begin{equation}
%TCIMACRO{\U{3bc} }%
%BeginExpansion
\mu
%EndExpansion
_{1}\Vert v\Vert _{N}^{2}+\Vert \nabla z_{m}\Vert _{2}^{2}\leq 
%TCIMACRO{\U{3bc} }%
%BeginExpansion
\mu
%EndExpansion
_{1}+0(1)  \tag{10}  \label{10}
\end{equation}

Now by Brezis-Lieb lemma, we get

\begin{equation*}
\lim_{m}\ \int_{M}\ v_{m}^{N}+z_{m}^{N}dv_{g}=\dint\limits_{M}v^{N}dv_{g}
\end{equation*}

i.e. 
\begin{equation*}
\lim_{m}\text{ }\Vert v_{m}\Vert _{N}^{N}-\Vert z_{m}\Vert _{N}^{N}=\Vert
v\Vert _{N}^{N}\text{.}
\end{equation*}

Hence 
\begin{equation*}
\Vert v_{m}\Vert _{N}^{N}+0(1)=\Vert z_{m}\Vert _{N}^{N}+\Vert v\Vert
_{N}^{N}\text{.}
\end{equation*}

By H\"{o}lder's inequality and $\int_{M}u_{m}^{N-2}v_{m}^{2}dv_{g}=1$, we get

$\ $%
\begin{equation*}
\ \Vert v_{m}\Vert _{N}^{N}\geq 1
\end{equation*}

i.e. $\ $%
\begin{equation*}
\int_{M}\ v^{N}+z_{m}^{N}\text{ }dv_{g}=\int_{M}\ v_{m}^{N}dv_{g}+0(1)\geq
1+0(1)\text{.}
\end{equation*}

Then%
\begin{equation*}
\left( \int_{M}v^{N}dv_{g}\right) ^{\frac{2}{N}}+\left(
\int_{M}z_{m}^{N}dv_{g}\right) ^{\frac{2}{N}}\geq 1+0(1)
\end{equation*}

i.e.

\begin{equation*}
\Vert z_{m}\Vert _{N}^{2}+\Vert v\Vert _{N}^{2}\geq 1+0(1)\text{.}
\end{equation*}

Now by Theorem \ref{th5} and the fact $\ z_{m}\rightarrow 0$\ \ strongly in $%
\ L^{2}$, we get

\begin{equation*}
\Vert z_{m}\Vert _{N}^{2}\leq (K^{2}+\varepsilon )\Vert \nabla z_{m}\Vert
_{2}^{2}+o(1)
\end{equation*}

\begin{equation*}
1+o(1)\leq \Vert z_{m}\Vert _{N}^{2}+\Vert v\Vert _{N}^{2}\leq \Vert v\Vert
_{N}^{2}+(K^{2}+\varepsilon )\Vert \nabla z_{m}\Vert _{2}^{2}+o(1)\text{.}
\end{equation*}

So we deduce

\begin{equation*}
1+o(1)\leq \Vert v\Vert _{N}^{2}+(K^{2}+\varepsilon )\Vert \nabla z_{m}\Vert
_{2}^{2}+o(1)
\end{equation*}

and from inequality (\ref{10}), we get

\begin{equation*}
\Vert \nabla z_{m}\Vert _{2}^{2}+%
%TCIMACRO{\U{3bc} }%
%BeginExpansion
\mu
%EndExpansion
_{1}\Vert v\Vert _{N}^{2}\leq 
%TCIMACRO{\U{3bc} }%
%BeginExpansion
\mu
%EndExpansion
_{1}((K^{2}+\varepsilon )\Vert \nabla z_{m}\Vert _{2}^{2}+\Vert v\Vert
_{N}^{2})+o(1)\text{.}
\end{equation*}%
So if $\mu _{1}K^{2}<1$, we get

\begin{equation*}
(1-%
%TCIMACRO{\U{3bc} }%
%BeginExpansion
\mu
%EndExpansion
_{1}(K^{2}+\varepsilon ))\Vert \nabla z_{m}\Vert _{2}^{2})\leq 0(1)
\end{equation*}

i.e. $\ v_{m}\rightarrow v$\ \ strongly in $H_{1}^{2}\left( M\right) $.

\textit{Step3}. We have

\begin{equation*}
\lim_{m}\int_{M}\left(
u_{m}^{N-2}v_{m}^{2}-u^{N-2}v^{2}+u_{m}^{N-2}v^{2}-u_{m}^{N-2}v^{2}\right)
dv_{g}
\end{equation*}

\begin{equation*}
=\lim_{m}\int \left(
u_{m}^{N-2}(v_{m}^{2}-v^{2})+(u_{m}^{N-2}-u^{N-2})v^{2}\right) dv_{g}\text{.}
\end{equation*}%
\newline

Now since $\ u_{m}\rightarrow u$\ \ a.e. so does $\ u_{m}^{N-2}\rightarrow u$%
\ $^{N-2}$\ \ and $\ \int_{M}u_{m}^{N-2}dv_{g}\leq c$,

hence $u_{m}^{N-2}$\ is bounded in $L^{N/(N-2\text{)}}$ and up to a
subsequence $u_{m}^{N-2}\rightarrow u$\ $^{N-2}$ \ \ weakly in $L^{N/(N-2%
\text{)}}$. Because of $v^{2}\in L^{\frac{N}{2}}(M)$, we have 
\begin{equation*}
\lim_{m}\int (u_{m}^{N-2}-u^{N-2})v^{2}dv_{g}=0\ 
\end{equation*}%
and by H\"{o}lder's inequality

\begin{equation*}
\lim_{m}\int u_{m}^{N-2}(v_{m}-v)^{2}dv_{g}\leq (\int
u_{m}^{N}dv_{g})^{(N-2)/N}(\int_{m}\left\vert v_{m}-v\right\vert
^{N}dv_{g})^{\frac{2}{N}}\leq 0\text{.}
\end{equation*}%
\ By the strong convergence of $v_{m}$\ $\ $in\ $\ L^{N}\left( M\right) $,
we get $\int_{M}u^{N-2}v^{2}dv_{g}=1$, then $v$\ \ and $u$\ \ are non
trivial functions.

\textit{Step4}.

Let $\overline{u}=av\in L_{+}^{N}\left( M\right) $\ with $a>0$\ a constant
such that $\dint\limits_{M}\overline{u}^{N}dv_{g}=1$\ with $v$\ a solution of

\begin{equation*}
L_{g}(v)=%
%TCIMACRO{\U{3bc} }%
%BeginExpansion
\mu
%EndExpansion
_{1}u^{N-2}v
\end{equation*}%
\ with the constraint%
\begin{equation*}
\int_{M}u^{N-2}v^{2}dv_{g}=1\text{.}
\end{equation*}%
\ We claim that $u=v$; indeed,

\begin{equation*}
%TCIMACRO{\U{3bc} }%
%BeginExpansion
\mu
%EndExpansion
_{1}\leq \frac{\int_{M}vL_{g}(v)dv_{g}}{\int_{M}\overline{u}^{N-2}v^{2}dv_{g}%
}
\end{equation*}%
\begin{equation*}
\leq \frac{\int_{M}vL_{g}(v)dv_{g}}{\int_{M}(av)^{N-2}v^{2}dv_{g}}=\frac{%
a^{2}%
%TCIMACRO{\U{3bc} }%
%BeginExpansion
\mu
%EndExpansion
_{1}\int_{M}u^{N-2}v^{2}dv_{g}}{\int_{M}\overline{u}^{N-2}(av)^{2}dv_{g}}
\end{equation*}%
and H\"{o}lder's inequality lead 
\begin{equation*}
\leq 
%TCIMACRO{\U{3bc} }%
%BeginExpansion
\mu
%EndExpansion
_{1}\int_{M}(u)^{N-2}(av)^{2}dv_{g}
\end{equation*}

\begin{equation*}
\leq 
%TCIMACRO{\U{3bc} }%
%BeginExpansion
\mu
%EndExpansion
_{1}(\int_{M}(u)^{N-2\frac{N}{N-2}})^{\frac{N-2}{N}}(\int_{M}(av)^{2\frac{N}{%
2}}dv_{g})^{\frac{2}{N}}\leq 
%TCIMACRO{\U{3bc} }%
%BeginExpansion
\mu
%EndExpansion
_{1}\text{.}
\end{equation*}

And since the equality in H\"{o}lder's inequality holds if 
\begin{equation*}
\overline{u}=u=av
\end{equation*}%
then\ $a=1$\ and 
\begin{equation*}
u=v\ \text{.}
\end{equation*}

Then $v$ satisfies \ $L_{g}v=%
%TCIMACRO{\U{3bc} }%
%BeginExpansion
\mu
%EndExpansion
_{1}v^{N-1}$ , by Theorem \ref{th6} we get \ $\ v=u\in H_{2}^{p}\left(
M\right) \subset C^{1-\left[ \frac{n}{p}\right] ,\beta }\left( M\right) $
with $\beta \in \left( 0,1\right) $\ and $v=u>0$\ ,

Resuming, we have%
\begin{equation*}
L_{g}(v)=%
%TCIMACRO{\U{3bc} }%
%BeginExpansion
\mu
%EndExpansion
_{1}v^{N-1}\ \text{,\ \ \ }\int_{M}v^{N}dv_{g}=1\text{ \ and }\ v=u\in
H_{2}^{p}\left( M\right) \subset C^{1-\left[ \frac{n}{p}\right] ,\beta
}\left( M\right)
\end{equation*}%
so the metric $\tilde{g}=u^{N-2}g$\ minimizes $%
%TCIMACRO{\U{3bc} }%
%BeginExpansion
\mu
%EndExpansion
_{1}$.
\end{proof}

\section{\ Yamabe\ conformal\ invariant\ with\ singularities}

\begin{theorem}
$If$\ $\ 
%TCIMACRO{\U{3bc} }%
%BeginExpansion
\mu
%EndExpansion
\geq 0$, then $%
%TCIMACRO{\U{3bc} }%
%BeginExpansion
\mu
%EndExpansion
_{1}=%
%TCIMACRO{\U{3bc} }%
%BeginExpansion
\mu
%EndExpansion
$
\end{theorem}

\begin{proof}
Step1

If $\ 
%TCIMACRO{\U{3bc} }%
%BeginExpansion
\mu
%EndExpansion
>0$. Let $v$\ \ such that $\ L_{g}(v)=%
%TCIMACRO{\U{3bc} }%
%BeginExpansion
\mu
%EndExpansion
_{1}v^{N-1}$\ and $\ \int_{M}v^{N}dv_{g}=1$\ then\qquad

\begin{equation*}
%TCIMACRO{\U{3bc} }%
%BeginExpansion
\mu
%EndExpansion
_{1}=\int_{M}vL_{g}(v)dv_{g}\geq c\Vert v\Vert _{H_{1}^{2}}
\end{equation*}

and $v$ \ in non trivial function then $\ 
%TCIMACRO{\U{3bc} }%
%BeginExpansion
\mu
%EndExpansion
_{1}>0$.\ On the other hand

\begin{equation*}
%TCIMACRO{\U{3bc} }%
%BeginExpansion
\mu
%EndExpansion
=\inf \frac{\int_{M}vL_{g}(v)dv_{g}}{(\int_{M}v^{N}dv_{g})^{\frac{2}{N}}}
\end{equation*}%
\begin{equation*}
\leq \int_{M}vL_{g}(v)dv_{g}=%
%TCIMACRO{\U{3bc} }%
%BeginExpansion
\mu
%EndExpansion
_{1}
\end{equation*}

and by Lemma \ref{lem6} , we get 
\begin{equation*}
%TCIMACRO{\U{3bc} }%
%BeginExpansion
\mu
%EndExpansion
_{1}=%
%TCIMACRO{\U{3bc}}%
%BeginExpansion
\mu%
%EndExpansion
\end{equation*}

Step2

If $\ 
%TCIMACRO{\U{3bc} }%
%BeginExpansion
\mu
%EndExpansion
=0$,\ Lemma \ref{lem6}$\ $implies that $\ 
%TCIMACRO{\U{3bc} }%
%BeginExpansion
\mu
%EndExpansion
_{1}\leq 0$\ , hence 
\begin{equation*}
%TCIMACRO{\U{3bc} }%
%BeginExpansion
\mu
%EndExpansion
_{1}=0\text{.}
\end{equation*}
\end{proof}

\section{Variational\ characterization\ of\ \ $\protect\mu _{2}$}

Let $\ [g]=\{u^{N-2}g$, $\ u\in H_{2}^{p}\left( M\right) $\ and $u>0\}$,\ we
define the second Yamabe invariant $%
%TCIMACRO{\U{3bc} }%
%BeginExpansion
\mu
%EndExpansion
_{2}$ as

\begin{equation*}
%TCIMACRO{\U{3bc} }%
%BeginExpansion
\mu
%EndExpansion
_{2}=\underset{\overline{g}\in \lbrack g]}{\inf }\lambda _{2,\overline{g}%
}Vol(M,\overline{g})^{^{2/n}}
\end{equation*}

or more explicitly

\begin{equation*}
%TCIMACRO{\U{3bc} }%
%BeginExpansion
\mu
%EndExpansion
_{2}=\underset{u\in H_{2}^{P},V\in Gr_{2}^{u}(H_{1}^{2}\left( M\right) )}{%
\inf }\underset{v\in V}{\text{ }\sup }\frac{\int_{M}vL_{g}(v)dv_{g}}{%
\int_{M}|u|^{N-2}v^{2}dv_{g}}(\dint\limits_{M}u^{N}dv_{g})^{\frac{2}{n}}
\end{equation*}

\begin{theorem}
\label{th10} \cite{1} On compact Riemannian manifold $\ (M,g)$\ of dimension 
$\ n\geq 3$, we have for all $\ v\in $\ $H_{1}^{2}\left( M\right) $\ and for
all $u\in L_{+}^{N}\left( M\right) $
\end{theorem}

\begin{equation*}
2^{\frac{2}{n}}\int_{M}|u|^{N-2}v^{2}dv_{g}\leq (K^{2}\int_{M}|\nabla
v|^{2}dv_{g}+\int_{M}B_{0}v^{2}dv_{g})(\dint\limits_{M}u^{N}dv_{g})^{\frac{2%
}{n}}
\end{equation*}

Or%
\begin{equation*}
2^{\frac{2}{n}}\int_{M}|u|^{N-2}v^{2}dv_{g}\leq 
%TCIMACRO{\U{3bc} }%
%BeginExpansion
\mu
%EndExpansion
_{1}(S_{n})(\int_{M}C_{n}|\nabla
v|^{2}+B_{0}v^{2}dv_{g})(\dint\limits_{M}u^{N}dv_{g})^{\frac{2}{n}}
\end{equation*}

\begin{theorem}
\label{th11} \cite{1} For any compact Riemannian manifold $\ (M,g)$ of
dimension $\ n\geq 3$, there exists $B_{0}>0$\ such that 
\begin{equation*}
%TCIMACRO{\U{3bc} }%
%BeginExpansion
\mu
%EndExpansion
_{1}(S_{n})=n(n-1)\omega _{n}^{2/n}\ =\underset{H_{1}^{2}}{\inf }\frac{%
\int_{M}\frac{4(n-1)}{(n-2)}|\nabla u|^{2}+B_{0}u^{2}dv_{g}}{%
(\int_{M}|u|^{N}dv_{g})^{2/N}}
\end{equation*}%
where $\ \omega _{n}$\ is the volume of the unit round sphere

or%
\begin{equation*}
(\int_{M}|u|^{N}dv_{g})^{2/N}\leq K^{2}\int_{M}|\nabla
u|^{2}dv_{g}+\int_{M}B_{0}u^{2}dv_{g}
\end{equation*}%
$K^{2}=%
%TCIMACRO{\U{3bc} }%
%BeginExpansion
\mu
%EndExpansion
_{1}(S_{n})^{-1}C_{n}$\ and $C_{n}=(4(n-1))/(n-2)$
\end{theorem}

\bigskip

\section{Properties\ of\ \ $\protect\mu _{2}$}

We know that $g\ \ $is smooth in the ball $B_{p}(\delta )$\ by assumption $%
(H)$, this assumption is sufficient to prove that Aubin's conjecture is
valid. The case $(M,g)$ is not conformally flat in a neighborhood of the
point $P$ \ and $n\geq 6$, let $\eta $\ is a cut-off function with support
in the ball $B_{p}(2\varepsilon )$ and $\eta =1$\ in $B_{p}(\varepsilon )$,
where $2\varepsilon \leq \delta $\ and

\begin{equation*}
v_{\varepsilon }(q)=(\frac{\varepsilon }{r^{2}+\varepsilon ^{2}})^{\frac{n-2%
}{2}}
\end{equation*}%
with $r=d(p,q).$We let $u_{\varepsilon }=\eta v_{\varepsilon }$ and define 
\begin{equation*}
Y(u)=\frac{\int_{M}\left( |\nabla u|^{2}+\frac{n-2}{4(n-1)}S_{g}u^{2}\right)
dv_{g}}{(\int_{M}|u|^{N}dv_{g})^{2/N}}\text{.}
\end{equation*}%
We obtain the following lemma

\begin{lemma}
\cite{1}%
\begin{equation*}
%TCIMACRO{\U{3bc} }%
%BeginExpansion
\mu
%EndExpansion
=Y(v_{\varepsilon })\leq \left\{ 
\begin{array}{c}
\{(K^{-2}-c|w(P)|^{2}\varepsilon ^{4}+0(\varepsilon ^{4})\text{ \ \ \ if}\
n>6 \\ 
K^{-2}-c|w(P)|^{2}\varepsilon ^{4}\log \frac{1}{\varepsilon }+0(\varepsilon
^{4})\ \ \text{if}\ \ n=6%
\end{array}%
\right.
\end{equation*}%
where $\left\vert w(P)\right\vert $is the norm of the Weyl tensor at the
point $P$\ and $c>0$.
\end{lemma}

\begin{theorem}
\label{th8} If $(M,g)$\ in not locally conformally flat and $\ n\geq 11$\
and $\ 
%TCIMACRO{\U{3bc} }%
%BeginExpansion
\mu
%EndExpansion
>0$\ , we find 
\begin{equation*}
%TCIMACRO{\U{3bc} }%
%BeginExpansion
\mu
%EndExpansion
_{2}<((%
%TCIMACRO{\U{3bc} }%
%BeginExpansion
\mu
%EndExpansion
^{\frac{n}{2}}+(K^{-2})^{\frac{n}{2}})^{\frac{2}{n}}
\end{equation*}%
\ and if $\ \ 
%TCIMACRO{\U{3bc} }%
%BeginExpansion
\mu
%EndExpansion
=0$\ , $n\geq 9$\ then 
\begin{equation*}
\ 
%TCIMACRO{\U{3bc} }%
%BeginExpansion
\mu
%EndExpansion
_{2}<K^{-2}
\end{equation*}
\end{theorem}

\begin{proof}
With the same method as in \cite{1}, this lemma follows from theorem\ref{th8}%
.
\end{proof}

\section{\protect\bigskip Existence\ of\ a\ minimizer to\ $\protect\mu _{2}$}

\begin{lemma}
\label{lem8} Assume that $v_{m}\rightarrow v$ weakly in\ $\ H_{1}^{2}\left(
M\right) $, $u_{m}$\ $\rightarrow u$\ weakly in $L^{N}\left( M\right) $ and
\ $\int_{M}u_{m}{}^{N-2}v_{m}{}^{2}dv_{g}=1$\ then
\end{lemma}

\begin{equation*}
\int_{M}u_{m}{}^{N-2}(v_{m}-v)^{2}dv_{g}=1-\int_{M}u^{N-2}v^{2}dv_{g}+o(1)
\end{equation*}

\bigskip

\begin{proof}
we have 
\begin{equation*}
\int_{M}u_{m}{}^{N-2}(v_{m}-v)^{2}dv_{g}
\end{equation*}

\begin{equation*}
=\int_{M}u_{m}{}^{N-2}v_{m}{}^{2}dv_{g}+\int_{M}u_{m}{}^{N-2}v^{2}dv_{g}-%
\int_{M}2u_{m}{}^{N-2}v_{m}vdv_{g}
\end{equation*}

\begin{equation}
=1+\int_{M}u_{m}{}^{N-2}v^{2}dv_{g}-\int_{M}2u_{m}{}^{N-2}v_{m}vdv_{g}\text{
.}  \tag{15}  \label{15}
\end{equation}

Now $\left( u_{m}{}^{N-2}\right) _{m}$ is bounded in $L^{\frac{N}{N-2}%
}\left( M\right) $ and $u_{m}{}^{N-2}\rightarrow u^{N-2}$a.e., then $%
u_{m}{}^{N-2}\rightarrow u^{N-2}$ weakly in $L^{\frac{N}{N-2}}\left(
M\right) $ and $\forall \phi \in L^{\frac{N}{2}}$ 
\begin{equation*}
\int_{M}\phi u_{m}{}^{N-2}dv_{g}\rightarrow \int_{M}\phi u^{N-2}dv_{g}
\end{equation*}%
in particular for $\phi =v^{2}$ 
\begin{equation*}
\int_{M}v^{2}u_{m}{}^{N-2}dv_{g}\rightarrow \int_{M}v^{2}u^{N-2}dv_{g}\text{.%
}
\end{equation*}

$\int_{M}u_{m}{}^{N-2}v_{m}dv_{g}$ is bounded in $L^{\frac{N}{N-1}}\left(
M\right) $, because of 
\begin{equation*}
\int_{M}u_{m}{}^{N-2\frac{N}{N-1}}v_{m}^{\frac{N}{N-1}}dv_{g}\leq
(\int_{M}u_{m}{}^{N}dv_{g})^{\frac{N-2}{N-1}}(\int_{M}v_{m}^{N}dv_{g})^{%
\frac{1}{N-1}}
\end{equation*}

and $u_{m}{}^{N-2}v_{m}\rightarrow u{}^{N-2}v$ a.e., then $%
u_{m}{}^{N-2}v_{m}\rightarrow u{}^{N-2}v$ weakly in $L^{\frac{N}{N-1}}\left(
M\right) $.

Hence

\begin{equation*}
\int_{M}u_{m}{}^{N-2}v_{m}vdv_{g}\rightarrow \int_{M}u{}^{N-2}v^{2}dv_{g}
\end{equation*}

and 
\begin{equation*}
\int_{M}u_{m}{}^{N-2}(v_{m}-v)^{2}dv_{g}=1-\int_{M}u^{N-2}v^{2}dv_{g}+o(1)%
\text{.}
\end{equation*}
\end{proof}

\begin{theorem}
If $\ \ 1-2^{-\frac{2}{n}}K^{2}%
%TCIMACRO{\U{3bc} }%
%BeginExpansion
\mu
%EndExpansion
_{2}>0$, then the generalized metric $u^{N-2}g$\ minimizes $%
%TCIMACRO{\U{3bc} }%
%BeginExpansion
\mu
%EndExpansion
_{2}$
\end{theorem}

\begin{proof}
\textit{Step1}.

We study a sequence of metrics $g_{m}=u_{m}^{N-2}g$\ with $u_{m}\in
H_{2}^{p}\left( M\right) $, $u_{m}>0$ which minimizes the infimum in the
definition of $\ 
%TCIMACRO{\U{3bc} }%
%BeginExpansion
\mu
%EndExpansion
_{2}$\ i.e. a sequence of metrics such that

\begin{equation*}
%TCIMACRO{\U{3bc} }%
%BeginExpansion
\mu
%EndExpansion
_{2}=\lim \lambda _{2,m}(Vol(M,g_{m})^{2/n}\text{.}
\end{equation*}%
Without loss generality, we may assume that $Vol(M,g_{m})=1$\ i.e. that $%
\int_{M}u_{m}^{N}dv_{g}=1.$In particular, the sequence of functions $\left(
u_{m}\right) _{m}$\ is bounded in $L^{N}\left( M\right) $\ and there exists $%
u\in L^{N}\left( M\right) $, $\ u\geq 0$\ \ such that $\ u_{m}\rightarrow u$%
\ weakly in $L^{N}$. We are going to prove that the generalized metric $\
u^{N-2}g$\ minimizes\ $\ 
%TCIMACRO{\U{3bc} }%
%BeginExpansion
\mu
%EndExpansion
_{2}$\ .Proposition \ref{prop2} , implies the existence of $\ v_{m},w_{m}\in
H_{1}^{2}\left( M\right) $, $v_{m}>0$\ such that 
\begin{equation*}
L_{g}(v_{m})=\lambda _{1,m}u_{m}^{N-2}v_{m}
\end{equation*}%
\ \ 
\begin{equation*}
L_{g}(w_{m})=\lambda _{2,m}u_{m}^{N-2}w_{m}
\end{equation*}%
\ And such that $\ $%
\begin{equation*}
\int_{M}u_{m}^{N-2}v_{m}^{2}dv_{g}=\int_{M}u_{m}^{N-2}w_{m}^{2}dv_{g}=1,%
\int_{M}u_{m}^{N-2}v_{m}w_{m}dv_{g}=0\text{.}
\end{equation*}%
The sequence $v_{m}$\ ,$w_{m}$\ \ is bounded in $H_{1}^{2}\left( M\right) $,
we can find $\ v,w\in H_{1}^{2}\left( M\right) $, $v\geq 0$ \ such that $%
v_{m}\rightarrow v$\ , $w_{m}\rightarrow w$\ \ weakly in $\ H_{1}^{2}\left(
M\right) $.Together with the weak convergence of $\ \left( u_{m}\right) $,
we get in weak sense%
\begin{equation*}
L_{g}(v)=\widehat{%
%TCIMACRO{\U{3bc} }%
%BeginExpansion
\mu
%EndExpansion
_{1}}u^{N-2}v\ 
\end{equation*}%
and 
\begin{equation*}
L_{g}(w)=%
%TCIMACRO{\U{3bc} }%
%BeginExpansion
\mu
%EndExpansion
_{2}u^{N-2}w
\end{equation*}%
where 
\begin{equation*}
\widehat{%
%TCIMACRO{\U{3bc} }%
%BeginExpansion
\mu
%EndExpansion
_{1}}=\lim \lambda _{1,m}\leq \ 
%TCIMACRO{\U{3bc} }%
%BeginExpansion
\mu
%EndExpansion
_{2}\text{.}\ 
\end{equation*}

\textit{Step2.}

Now we show $\ v_{m}\rightarrow v$\ , $w_{m}\rightarrow w$\ \ strongly in $%
H_{1}^{2.}\left( M\right) $. Applying Theorem \ref{th10} to the sequence $%
v_{m}-v$, we get

\begin{equation*}
\int_{M}|u_{m}|^{N-2}(v_{m}-v)^{2}dv_{g}\leq (2^{-\frac{2}{n}%
}K^{2}\int_{M}|\nabla
(v_{m}-v)|^{2}dv_{g}+\int_{M}B_{0}(v_{m}-v)^{2}dv_{g})(\dint%
\limits_{M}u^{N}dv_{g})^{\frac{2}{n}}
\end{equation*}%
and since $v_{m}\rightarrow v$\ \ strongly in $L^{2}$\ ,

\begin{equation*}
\int_{M}|u_{m}|^{N-2}(v_{m}-v)^{2}dv_{g}\leq (2^{-\frac{2}{n}%
}K^{2}\int_{M}|\nabla (v_{m}-v)|^{2}dv_{g}+o(1)
\end{equation*}%
\begin{equation*}
\leq (2^{-\frac{2}{n}}K^{2}\int_{M}|\nabla (v_{m})|^{2}+\left\vert \nabla
v\right\vert ^{2}-2\nabla v_{m}\nabla vdv_{g}+o(1)\text{.}
\end{equation*}%
By the weak convergence of $\left( v_{m}\right) $\ , $\int_{M}\nabla
v_{m}\nabla vdv_{g}=\int_{M}\left\vert \nabla v\right\vert ^{2}dv_{g}+o(1)$

\begin{equation*}
\int_{M}|u_{m}|^{N-2}(v_{m}-v)^{2}dv_{g}\leq (2^{-\frac{2}{n}%
}K^{2}\int_{M}|\nabla (v_{m})|^{2}-\left\vert \nabla v\right\vert
^{2}dv_{g}+o(1)
\end{equation*}%
and since 
\begin{equation*}
\ \int_{M}\frac{n-2}{4(n-1)}S_{g}v_{m}{}^{2}dv_{g}=\int_{M}\frac{n-2}{4(n-1)}%
S_{g}v^{2}dv_{g}+0(1)
\end{equation*}%
we get

\begin{equation*}
\int_{M}|u_{m}|^{N-2}(v_{m}-v)^{2}dv_{g}\leq 2^{-\frac{2}{n}%
}K^{2}(\int_{M}|\nabla (v_{m})|^{2}-\left\vert \nabla v\right\vert ^{2}dv_{g}
\end{equation*}

\begin{equation*}
+\int_{M}\frac{n-2}{4(n-1)}S_{g}(v_{m}^{2}{}-v^{2})dv_{g})+o(1)\leq 2^{-%
\frac{2}{n}}K^{2}(\int_{M}v_{m}L_{g}(v_{m})-vL_{g}(v)dv_{g})+o(1)
\end{equation*}

\begin{equation*}
\leq 2^{-\frac{2}{n}}K^{2}(\lambda _{1,m}-\widehat{%
%TCIMACRO{\U{3bc} }%
%BeginExpansion
\mu
%EndExpansion
_{1}}\int_{M}u^{N-2}v^{2}dv_{g})+o(1)
\end{equation*}%
By the fact \ $\widehat{%
%TCIMACRO{\U{3bc} }%
%BeginExpansion
\mu
%EndExpansion
_{1}}=\lim \lambda _{1,m}\leq $\ $%
%TCIMACRO{\U{3bc} }%
%BeginExpansion
\mu
%EndExpansion
_{2}$

\begin{equation*}
\leq 2^{-\frac{2}{n}}K^{2}%
%TCIMACRO{\U{3bc} }%
%BeginExpansion
\mu
%EndExpansion
_{2}(1-\int_{M}u^{N-2}v^{2}dv_{g})+o(1)
\end{equation*}%
Then

\begin{equation*}
\int_{M}|u_{m}|^{N-2}(v_{m}-v)^{2}dv_{g}\leq 2^{-\frac{2}{n}}K^{2}%
%TCIMACRO{\U{3bc} }%
%BeginExpansion
\mu
%EndExpansion
_{2}(1-\int_{M}u^{N-2}v^{2})dv_{g}+o(1)
\end{equation*}%
Now using the weak convergence of $\left( v_{m}\right) $\ in $\
H_{1}^{2}\left( M\right) $ and the weak convergence of $\left( u_{m}\right) $%
\ \ in $L^{N}\left( M\right) $, we infer by Lemma \ref{lem8} that

\begin{equation*}
\int_{M}|u_{m}|^{N-2}(v_{m}-v)^{2}dv_{g}=1-\int_{M}u^{N-2}v^{2}dv_{g}+o(1)
\end{equation*}%
then

\begin{equation*}
1-\int_{M}u^{N-2}v^{2}dv_{g}\leq 2^{-\frac{2}{n}}K^{2}%
%TCIMACRO{\U{3bc} }%
%BeginExpansion
\mu
%EndExpansion
_{2}(1-\int_{M}u^{N-2}v^{2})dv_{g}+o(1)
\end{equation*}

and%
\begin{equation*}
1-2^{-\frac{2}{n}}K^{2}%
%TCIMACRO{\U{3bc} }%
%BeginExpansion
\mu
%EndExpansion
_{2}\leq (1-2^{-\frac{2}{n}}K^{2}%
%TCIMACRO{\U{3bc} }%
%BeginExpansion
\mu
%EndExpansion
_{2})\int_{M}u^{N-2}v^{2}dv_{g}+o(1)\text{.}
\end{equation*}%
So if $\ \ 1-2^{-\frac{2}{n}}K^{2}%
%TCIMACRO{\U{3bc} }%
%BeginExpansion
\mu
%EndExpansion
_{2}>0$\ then 
\begin{equation*}
\int_{M}u^{N-2}v^{2}dv_{g}\geq 1\text{.}
\end{equation*}%
and by Fatou's lemma, we obtain

\begin{equation*}
\int_{M}u^{N-2}v^{2}dv_{g}\leq \underline{\lim }\
\int_{M}u_{m}^{N-2}v_{m}^{2}dv_{g}=1\text{.}
\end{equation*}%
We find that 
\begin{equation}
\int_{M}u^{N-2}v^{2}dv_{g}=1\text{.}  \tag{16}  \label{16}
\end{equation}%
So $u$ and $v$\ are not trivial.

Moreover%
\begin{equation*}
\int_{M}\left\vert \nabla (v_{m}-v)\right\vert ^{2}dv_{g}=\int_{M}\left(
|\nabla (v_{m})|^{2}+\left\vert \nabla v\right\vert ^{2}-2\nabla v_{m}\nabla
v\right) dv_{g}
\end{equation*}

\begin{equation*}
=\int_{M}|\nabla (v_{m})|^{2}-\left\vert \nabla v\right\vert ^{2}dv_{g}+o(1)
\end{equation*}

and since $\ \int_{M}S_{g}\left( v_{m}{}^{2}-v^{2}\right) dv_{g}=o(1)$, we
get%
\begin{equation*}
\int_{M}\left\vert \nabla (v_{m}-v)\right\vert
^{2}dv_{g}=\int_{M}v_{m}L_{g}(v_{m})-vL_{g}(v)dv_{g}+o(1)
\end{equation*}

\begin{equation*}
\leq 
%TCIMACRO{\U{3bc} }%
%BeginExpansion
\mu
%EndExpansion
_{2}(1-\int_{M}u^{N-2}v^{2}dv_{g})+o(1)
\end{equation*}%
Then, by relation (\ref{16})

\begin{equation*}
\int_{M}\left\vert \nabla (v_{m}-v)\right\vert ^{2}dv_{g}=o(1)
\end{equation*}%
and $v_{m}\rightarrow v$\ strongly in $H_{1}^{2}\left( M\right) .$The same
argument holds with $\left( w_{m}\right) $, hence $w_{m}\rightarrow w$\
strongly in $H_{1}^{2}\left( M\right) $ and $\int_{M}u^{N-2}w^{2}dv_{g}=1$.\ 

To show that $\int_{M}u^{N-2}vwdv_{g}=0$, first writing and using H\^{o}%
lder's inequality, we get

\begin{equation*}
\int_{M}\left( u_{m}^{N-2}v_{m}w_{m}-u^{N-2}vw\right) dv_{g}=\int_{M}\left(
u_{m}^{N-2}v_{m}w_{m}-u_{m}^{N-2}vw_{m}+u_{m}^{N-2}vw_{m}-u^{N-2}vw\right)
dv_{g}
\end{equation*}

\begin{equation*}
=\int_{M}u_{m}^{N-2}(v_{m}-v)w_{m}dv_{g}+\int_{M}\left(
u_{m}^{N-2}vw_{m}-u^{N-2}vw\right) dv_{g}
\end{equation*}

\begin{equation*}
=\int_{M}u_{m}^{\frac{N-2}{2}}w_{m}[u_{m}^{\frac{N-2}{2}}(v_{m}-v)]dv_{g}+%
\int_{M}\left( u_{m}^{N-2}vw_{m}-u^{N-2}vw\right) dv_{g}
\end{equation*}

\begin{equation*}
\leq \left( \int_{M}u_{m}^{N-2}w_{m}^{2}dv_{g}\right) ^{\frac{1}{2}}\left(
\int_{M}u_{m}^{N-2}(v_{m}-v)^{2}dv_{g}\right) ^{^{\frac{1}{2}%
}}+\int_{M}\left( u_{m}^{N-2}vw_{m}-u^{N-2}vw\right) dv_{g}
\end{equation*}

\begin{equation*}
\leq \left( \int_{M}u_{m}^{N-2}(v_{m}-v)^{2}dv_{g}\right) ^{\frac{1}{2}%
}+\int_{M}\left( u_{m}^{N-2}vw_{m}-u^{N-2}vw\right) dv_{g}
\end{equation*}

\begin{equation*}
\leq \left[ \left( \int_{M}u_{m}^{N-2\frac{N}{N-2}}dv_{g}\right) ^{\frac{N-2%
}{N}}\left( \int_{M}\left\vert v_{m}-v\right\vert ^{N}dv_{g}\right) ^{\frac{2%
}{N}}\right] ^{\frac{1}{2}}+\int_{M}\left(
u_{m}^{N-2}vw_{m}-u^{N-2}vw\right) dv_{g}
\end{equation*}

\begin{equation*}
\leq \left( \int_{M}\left\vert v_{m}-v\right\vert ^{N}dv_{g}\right) ^{\frac{1%
}{N}}+\int_{M}\left( u_{m}^{N-2}vw_{m}-u^{N-2}vw\right) dv_{g}
\end{equation*}

\begin{equation*}
\leq \left( \int_{M}\left\vert v_{m}-v\right\vert ^{N}dv_{g}\right) ^{\frac{1%
}{N}}+\int_{M}\left(
u_{m}^{N-2}vw_{m}-u_{m}^{N-2}vw+u_{m}^{N-2}vw-u^{N-2}vw\right) dv_{g}
\end{equation*}

\begin{equation*}
\leq \left( \int_{M}\left\vert v_{m}-v\right\vert ^{N}dv_{g}\right) ^{\frac{1%
}{N}}+\int_{M}\left( u_{m}^{N-2}v(w_{m}-w)+(u_{m}^{N-2}-u^{N-2})vw\right)
dv_{g}
\end{equation*}

\begin{equation*}
\leq \left( \int_{M}\left\vert v_{m}-v\right\vert ^{N}dv_{g}\right) ^{\frac{1%
}{N}}+\int_{M}\left( (u_{m}^{\frac{N-2}{2}}v)(u_{m}^{\frac{N-2}{2}%
}(w_{m}-w))+(u_{m}^{N-2}-u^{N-2})vw\right) dv_{g}
\end{equation*}

\begin{equation*}
\leq \left( \int_{M}\left\vert v_{m}-v\right\vert ^{N}dv_{g}\right) ^{\frac{1%
}{N}}+\left( \int_{M}u_{m}^{N-2}v^{2}dv_{g}\right) ^{\frac{1}{2}}\left(
\int_{M}u_{m}^{N-2}(w_{m}-w)^{2}dv_{g}\right) ^{\frac{1}{2}%
}+\int_{M}(u_{m}^{N-2}-u^{N-2})vwdv_{g}
\end{equation*}

\begin{equation*}
\leq \left( \int_{M}\left\vert v_{m}-v\right\vert ^{N}dv_{g}\right) ^{\frac{1%
}{N}}+\left( \int_{M}u_{m}^{N-2}v^{2}dv_{g}\right) ^{\frac{1}{2}}\left(
\int_{M}\left\vert w_{m}-w\right\vert )^{N}dv_{g}\right) ^{\frac{1}{N}%
}+\int_{M}(u_{m}^{N-2}-u^{N-2})vwdv_{g}\text{.}
\end{equation*}

Now noting that \ 
\begin{equation*}
\int_{M}u_{m}^{N-2}v^{2}dv_{g}\leq (\int_{M}u_{m}^{N}dv_{g})^{\frac{N-2}{2}%
}(\int_{M}v^{N}dv_{g})^{\frac{2}{N}}<+\infty
\end{equation*}

and taking account of $u_{m}^{N-2}\rightarrow u^{N-2}$ weakly in $L^{\frac{N%
}{N-2}}\left( M\right) $ and the fact that $vw\in L^{\frac{N}{2}}\left(
M\right) $, we deduce%
\begin{equation*}
\int_{M}(u_{m}^{N-2}-u^{N-2})vwdv_{g}\rightarrow 0
\end{equation*}%
\ hence 
\begin{equation*}
\int_{M}u^{N-2}vwdv_{g}=0\text{.}
\end{equation*}

Consequently the generalized metric $u^{N-2}g$\ minimizes $%
%TCIMACRO{\U{3bc} }%
%BeginExpansion
\mu
%EndExpansion
_{2}$.
\end{proof}

\bigskip

\begin{theorem}
If $%
%TCIMACRO{\U{3bc} }%
%BeginExpansion
\mu
%EndExpansion
_{2}<K^{-2}$, then generalized metric $u^{N-2}g$\ minimizes $%
%TCIMACRO{\U{3bc} }%
%BeginExpansion
\mu
%EndExpansion
_{2}$
\end{theorem}

\begin{proof}
\textit{Step1.}

We study a sequence of metrics $g_{m}=u_{m}^{N-2}g$\ with $u_{m}\in
H_{2}^{p}\left( M\right) $, $u_{m}>0$ which attained $\ 
%TCIMACRO{\U{3bc} }%
%BeginExpansion
\mu
%EndExpansion
_{2}$\ i.e. a sequence of metrics such that

\begin{equation*}
%TCIMACRO{\U{3bc} }%
%BeginExpansion
\mu
%EndExpansion
_{2}=\lim_{m}\lambda _{2,m}(Vol(M,g_{m})^{2/n}\text{.}
\end{equation*}%
Without loss of generality, we may assume that $Vol(M,g_{m})=1$\ i.e. $%
\dint\limits_{M}u_{m}^{N}dv_{g}=1.$In particular, the sequence $\left(
u_{m}\right) _{m}$\ is bounded in $L^{N}$\ $\left( M\right) $and there
exists $u\in L^{N}\left( M\right) $, $\ u\geq 0$\ \ such that $\
u_{m}\rightarrow u$\ weakly in $L^{N}\left( M\right) $.We are going to prove
that the metric $\ u^{N-2}g$\ minimizes\ $\ 
%TCIMACRO{\U{3bc} }%
%BeginExpansion
\mu
%EndExpansion
_{2}$. Proposition \ref{prop2} and Theorem \ref{6} imply the existence of $\
v_{m},$ $w_{m}\in C^{1-\left[ \frac{n}{p}\right] ,\beta }$, with $\beta \in
\left( 0,1\right) \left( M\right) $, $v_{m}>0$\ such that 
\begin{equation*}
L_{g}(v_{m})=\lambda _{1,m}u_{m}^{N-2}v_{m}
\end{equation*}%
\ \ 
\begin{equation*}
L_{g}(w_{m})=\lambda _{2,m}u_{m}^{N-2}w_{m}
\end{equation*}%
\ and $\ $%
\begin{equation*}
\int_{M}u_{m}^{N-2}v_{m}^{2}dv_{g}=\int_{M}u_{m}^{N-2}w_{m}^{2}dv_{g}=1\text{%
, }\int_{M}u_{m}^{N-2}v_{m}w_{m}dv_{g}=0\text{.}
\end{equation*}%
The sequences $\left( v_{m}\right) _{m}$and $\left( w_{m}\right) _{m}$\ \
are bounded in $H_{1}^{2}$, we can find $\ v,w\in H_{1}^{2}$ with $v\geq 0$
\ such that $v_{m}\rightarrow v$, $w_{m}\rightarrow w$\ \ weakly in $\
H_{1}^{2}$.Together with the weak convergence of $\ \left( u_{m}\right) _{m}$%
, we get in the weak sense%
\begin{equation*}
L_{g}(v)=\widehat{%
%TCIMACRO{\U{3bc} }%
%BeginExpansion
\mu
%EndExpansion
_{1}}u^{N-2}v\ 
\end{equation*}%
and%
\begin{equation*}
L_{g}(w)=%
%TCIMACRO{\U{3bc} }%
%BeginExpansion
\mu
%EndExpansion
_{2}u^{N-2}w
\end{equation*}%
where 
\begin{equation*}
\widehat{%
%TCIMACRO{\U{3bc} }%
%BeginExpansion
\mu
%EndExpansion
_{1}}=\lim \lambda _{1,m}\leq \ 
%TCIMACRO{\U{3bc} }%
%BeginExpansion
\mu
%EndExpansion
_{2}\text{.}\ 
\end{equation*}%
\textit{Step2}.

Now we are going to show that $v_{m}\rightarrow v$\ , $w_{m}\rightarrow w$\
\ strongly in $H_{1}^{2.}$.

By H\"{o}lder's inequality, Theorem \ref{th5}, strong convergence of $v_{m}$
in $L^{2.}$,we get 
\begin{equation*}
\int_{M}|u_{m}|^{N-2}(v_{m}-v)^{2}dv_{g}\leq \Vert v_{m}-v\Vert _{N}^{2}\leq
K^{2}\Vert \nabla (v_{m}-v)\Vert _{2}^{2}+o(1)
\end{equation*}

\begin{equation*}
\leq K^{2}\int_{M}|\nabla (v_{m})|^{2}+\left\vert \nabla v\right\vert
^{2}-2\nabla v_{m}\nabla vdv_{g}+o(1)
\end{equation*}

\begin{equation*}
\leq K^{2}\int_{M}|\nabla (v_{m})|^{2}-\left\vert \nabla v\right\vert
^{2}dv_{g}+o(1)
\end{equation*}

\begin{equation*}
\leq K^{2}\int_{M}v_{m}L_{g}(v_{m})-vL_{g}(v)dv_{g}+o(1)
\end{equation*}

\begin{equation*}
\leq K^{2}%
%TCIMACRO{\U{3bc} }%
%BeginExpansion
\mu
%EndExpansion
_{2}(1-\int_{M}u^{N-2}v^{2}dv_{g})+o(1)
\end{equation*}

and with Lemma \ref{lem8} 
\begin{equation*}
\int_{M}|u_{m}|^{N-2}(v_{m}-v)^{2}dv_{g}=1-\int_{M}u^{N-2}v^{2}dv_{g}+o(1)
\end{equation*}

then%
\begin{equation*}
1-\int_{M}u^{N-2}v^{2}dv_{g}\leq K^{2}%
%TCIMACRO{\U{3bc} }%
%BeginExpansion
\mu
%EndExpansion
_{2}(1-\int_{M}u^{N-2}v^{2}dv_{g})+o(1)
\end{equation*}

i.e%
\begin{equation*}
1-K^{2}%
%TCIMACRO{\U{3bc} }%
%BeginExpansion
\mu
%EndExpansion
_{2}\leq (1-K^{2}%
%TCIMACRO{\U{3bc} }%
%BeginExpansion
\mu
%EndExpansion
_{2})\int_{M}u^{N-2}v^{2}dv_{g}
\end{equation*}

so if $1-K^{2}%
%TCIMACRO{\U{3bc} }%
%BeginExpansion
\mu
%EndExpansion
_{2}>0$ , 
\begin{equation*}
\int_{M}u^{N-2}v^{2}dv_{g}\geq 1\text{.}
\end{equation*}%
On the other hand since by Fatou's lemma

\begin{equation*}
\int_{M}u^{N-2}v^{2}dv_{g}\leq \underline{\lim }\
\int_{M}u_{m}^{N-2}v_{m}^{2}dv_{g}=1\text{.}
\end{equation*}%
Then%
\begin{equation*}
\int_{M}u^{N-2}v^{2}dv_{g}=1.
\end{equation*}
and

\begin{equation*}
\int_{M}\left\vert \nabla (v_{m}-v)\right\vert ^{2}dv_{g}=o(1)
\end{equation*}

Hence $v_{m}\rightarrow v$\ \ strongly in $H_{1}^{2.}\subset L^{N}$.

The same conclusion also holds for $\left( w_{m}\right) _{m}$.
\end{proof}

\begin{lemma}
\label{lem9} Let $\ u\in L^{N}$ with $\dint\limits_{M}u^{N}dv_{g}=1$ and $z$ 
$,w$ \ nonnegative functions in $H_{1}^{2}$satisfying 
\begin{equation}
\int_{M}wL_{g}(w)dv_{g}\leq 
%TCIMACRO{\U{3bc} }%
%BeginExpansion
\mu
%EndExpansion
_{2}\int_{M}u^{N-2}w^{2}dv_{g}  \tag{20}  \label{20}
\end{equation}%
and%
\begin{equation}
\int_{M}zL_{g}(z)dv_{g}\leq 
%TCIMACRO{\U{3bc} }%
%BeginExpansion
\mu
%EndExpansion
_{2}\int_{M}u^{N-2}z^{2}dv_{g}  \tag{21}  \label{21}
\end{equation}%
And suppose that $(M-z^{-1}(0))\cap (M-w^{-1}(0))$\ has measure zero. Then $%
u $ \ is a linear combination of $\ z$\ \ and $\ w$\ \ and we have equality
in (\ref{20}) and (\ref{21}).
\end{lemma}

\begin{proof}
\bigskip The proof is the same as that of Aummann and Humbert in \cite{1}$.$
\end{proof}

\begin{theorem}
If a generalized metric $\ u^{N-2}g$\ minimizes $%
%TCIMACRO{\U{3bc} }%
%BeginExpansion
\mu
%EndExpansion
_{2}$, then there exist a nodal solution $w\in H_{2}^{p}\subset
C^{1-[n/p],\beta }$

\ of equation 
\begin{equation}
L_{g}(w)=%
%TCIMACRO{\U{3bc} }%
%BeginExpansion
\mu
%EndExpansion
_{2}u^{N-2}w\   \tag{22}  \label{22}
\end{equation}%
More over there exist$\ a,b>0$\ such that 
\begin{equation*}
u=aw_{+}+bw_{-}
\end{equation*}
\end{theorem}

With $w_{+}=sup(w,0)$\ \ and $w_{-}=sup(-w,0$) .

\begin{proof}
\textit{Step1}. Applying Lemma \ref{lem9} to $w_{+}=sup(w,0)$\ \ and $%
w_{-}=sup(-w,0)$, we get the existence of $\ a,b>0$\ such that 
\begin{equation*}
u=aw_{+}+bw_{-}\text{.}
\end{equation*}%
Now by Lemma \ref{lem7}, $w_{+},w_{-}\in L^{\infty }$ i.e. $\ u\in L^{\infty
}$\ , $u^{N-2}\in L^{\infty }$, then 
\begin{equation*}
h=S_{g}-%
%TCIMACRO{\U{3bc} }%
%BeginExpansion
\mu
%EndExpansion
_{2}u^{N-2}\in L^{p}
\end{equation*}%
\ and from Theorem \ref{th6}, we obtain%
\begin{equation*}
w\in H_{2}^{p}\subset C^{1-[n/p],\beta }\text{.}
\end{equation*}

\textit{Step 2.} If $\ 
%TCIMACRO{\U{3bc} }%
%BeginExpansion
\mu
%EndExpansion
_{2}=%
%TCIMACRO{\U{3bc} }%
%BeginExpansion
\mu
%EndExpansion
_{1}$, we see that $|w|$\ is a minimizer to the functional associated to $%
%TCIMACRO{\U{3bc} }%
%BeginExpansion
\mu
%EndExpansion
_{1}$, then $|w|$\ satisfies the same equation as $v$\ \ and Theorem \ref%
{th6} shows that $|w|=w\in H_{2}^{p}$\ $\subset \subset C^{1-[n/p],\beta }$\
that is $|w|>0$\ everywhere, which contradicts the condition $(9)$ in
Proposition \ref{prop2} , then 
\begin{equation*}
%TCIMACRO{\U{3bc} }%
%BeginExpansion
\mu
%EndExpansion
_{2}>%
%TCIMACRO{\U{3bc} }%
%BeginExpansion
\mu
%EndExpansion
_{1}\text{.}\ 
\end{equation*}

\textit{Step3.} The solution $w$ \ of the equation (\ref{22}) changes sign.
Since if it does not, we may assume that $\ w\geq 0$, by step2 \ the
inequality in (\ref{20})\ is strict and by Lemma (\ref{lem9}) we have the
equality: a contradiction.
\end{proof}

\begin{remark}
Step1 shows that $u$ is not necessarily in $H_{2}^{p}\left( M\right) $ and
by the way the minimizing metric is not in $H_{2}^{p}(M,T^{\ast }M\otimes
T^{\ast }M)$ contrary to the Yamabe invariant with singularities.
\end{remark}

\end{document}